\documentclass[12pt]{amsart}       
\usepackage{txfonts}
\usepackage{amssymb}
\usepackage{eucal}
\usepackage{graphicx}
\usepackage{amsmath}
\usepackage{amscd}
\usepackage[all]{xy}           
\usepackage{amsfonts,latexsym}
\usepackage{xspace}
\usepackage{epsfig}
\usepackage{float}
\usepackage{color}
\usepackage{fancybox}
\usepackage{colordvi}
\usepackage{multicol}
\usepackage{colordvi}
\usepackage[colorlinks,final,backref=page,hyperindex,hypertex]{hyperref}
\usepackage[active]{srcltx} 

\usepackage{graphicx}

\newtheorem{theorem}{Theorem}[section]

\theoremstyle{definition}

\newtheorem{lemma}[theorem]{Lemma}
\newtheorem{corollary}[theorem]{Corollary}
\newtheorem{prop-def}{Proposition-Definition}[section]
\newtheorem{coro-def}{Corollary-Definition}[section]

\newtheorem{remark}[theorem]{Remark}

\newtheorem{example}[theorem]{Example}

\newcommand{\nc}{\newcommand}
\nc{\tred}[1]{\textcolor{red}{#1}}
\nc{\tblue}[1]{\textcolor{blue}{#1}}
\nc{\tgreen}[1]{\textcolor{green}{#1}}
\nc{\tpurple}[1]{\textcolor{purple}{#1}}
\nc{\btred}[1]{\textcolor{red}{\bf #1}}
\nc{\btblue}[1]{\textcolor{blue}{\bf #1}}
\nc{\btgreen}[1]{\textcolor{green}{\bf #1}}
\nc{\btpurple}[1]{\textcolor{purple}{\bf #1}}
\nc{\NN}{{\mathbb N}}
\nc{\ncsha}{{\mbox{\cyr X}^{\mathrm NC}}} \nc{\ncshao}{{\mbox{\cyr
X}^{\mathrm NC}_0}}

\newcommand{\efootnote}[1]{}

\renewcommand{\textbf}[1]{}

\newcommand{\delete}[1]{}

\nc{\mlabel}[1]{\label{#1}}  
\nc{\mcite}[1]{\cite{#1}}  
\nc{\mref}[1]{\ref{#1}}  
\nc{\mbibitem}[1]{\bibitem{#1}} 

\delete{
\nc{\mlabel}[1]{\label{#1}{\hfill \hspace{1cm}{\bf{{\ }\hfill(#1)}}}}
\nc{\mcite}[1]{\cite{#1}{{\bf{{\ }(#1)}}}}  
\nc{\mref}[1]{\ref{#1}{{\bf{{\ }(#1)}}}}  
\nc{\mbibitem}[1]{\bibitem[\bf #1]{#1}} 
}

\nc{\opa}{\ast} \nc{\opb}{\odot} \nc{\op}{\bullet} \nc{\pa}{\frakL}
\nc{\arr}{\rightarrow} \nc{\lu}[1]{(#1)} \nc{\mult}{\mrm{mult}}
\nc{\diff}{\mathfrak{Diff}}
\nc{\opc}{\sharp}\nc{\opd}{\natural}
\nc{\ope}{\circ}
\nc{\dpt}{\mathrm{d}}
\nc{\hck}{H_{RT}}
\nc{\vdf}{\calf}
\nc{\ldf}{\calf_\ell}
\nc{\hlf}{H_\ell}
\nc{\onek}{\mathbf{1}_\bfk}
\nc{\diam}{alternating\xspace}
\nc{\Diam}{Alternating\xspace}
\nc{\cdiam}{canonical alternating\xspace}
\nc{\Cdiam}{Canonical alternating\xspace}
\nc{\AW}{\mathcal{A}}

\nc{\ari}{\mathrm{ar}}

\nc{\lef}{\mathrm{lef}}

\nc{\Sh}{\mathrm{ST}}

\nc{\Cr}{\mathrm{Cr}}

\nc{\st}{{Schr\"oder tree}\xspace}
\nc{\sts}{{Schr\"oder trees}\xspace}

\nc{\vertset}{\Omega} 

\nc{\assop}{\quad \begin{picture}(5,5)(0,0)
\line(-1,1){10}
\put(-2.2,-2.2){$\bullet$}
\line(0,-1){10}\line(1,1){10}
\end{picture} \quad \smallskip}

\nc{\operator}{\begin{picture}(5,5)(0,0)
\line(0,-1){6}
\put(-2.6,-1.8){$\bullet$}
\line(0,1){9}
\end{picture}}

\nc{\idx}{\begin{picture}(6,6)(-3,-3)
\put(0,0){\line(0,1){6}}
\put(0,0){\line(0,-1){6}}
\end{picture}}

\nc{\pb}{{\mathrm{pb}}}
\nc{\Lf}{{\mathrm{Lf}}}

\nc{\lft}{{left tree}\xspace}
\nc{\lfts}{{left trees}\xspace}

\nc{\fat}{{fundamental averaging tree}\xspace}

\nc{\fats}{{fundamental averaging trees}\xspace}
\nc{\avt}{\mathrm{Avt}}

\nc{\rass}{{\mathit{RAss}}}

\nc{\aass}{{\mathit{AAss}}}

\nc{\vin}{{\mathrm Vin}}    
\nc{\lin}{{\mathrm Lin}}    
\nc{\inv}{\mathrm{I}n}
\nc{\gensp}{V} 
\nc{\genbas}{\mathcal{V}} 
\nc{\bvp}{V_P}     
\nc{\gop}{{\,\omega\,}}     

\nc{\bin}[2]{ (_{\stackrel{\scs{#1}}{\scs{#2}}})}  
\nc{\binc}[2]{ \left (\!\! \begin{array}{c} \scs{#1}\\
    \scs{#2} \end{array}\!\! \right )}  
\nc{\bincc}[2]{  \left ( {\scs{#1} \atop
    \vspace{-1cm}\scs{#2}} \right )}  
\nc{\bs}{\bar{S}} \nc{\cosum}{\sqsubset} \nc{\la}{\longrightarrow}
\nc{\rar}{\rightarrow} \nc{\dar}{\downarrow} \nc{\dprod}{**}
\nc{\dap}[1]{\downarrow \rlap{$\scriptstyle{#1}$}}
\nc{\md}{\mathrm{dth}} \nc{\uap}[1]{\uparrow
\rlap{$\scriptstyle{#1}$}} \nc{\defeq}{\stackrel{\rm def}{=}}
\nc{\disp}[1]{\displaystyle{#1}} \nc{\dotcup}{\
\displaystyle{\bigcup^\bullet}\ } \nc{\gzeta}{\bar{\zeta}}
\nc{\hcm}{\ \hat{,}\ } \nc{\hts}{\hat{\otimes}}
\nc{\barot}{{\otimes}} \nc{\free}[1]{\bar{#1}}
\nc{\uni}[1]{\tilde{#1}} \nc{\hcirc}{\hat{\circ}} \nc{\lleft}{[}
\nc{\lright}{]} \nc{\lc}{\lfloor} \nc{\rc}{\rfloor}
\nc{\curlyl}{\left \{ \begin{array}{c} {} \\ {} \end{array}
    \right .  \!\!\!\!\!\!\!}
\nc{\curlyr}{ \!\!\!\!\!\!\!
    \left . \begin{array}{c} {} \\ {} \end{array}
    \right \} }
\nc{\longmid}{\left | \begin{array}{c} {} \\ {} \end{array}
    \right . \!\!\!\!\!\!\!}
\nc{\onetree}{\bullet} \nc{\ora}[1]{\stackrel{#1}{\rar}}
\nc{\ola}[1]{\stackrel{#1}{\la}}
\nc{\ot}{\otimes} \nc{\mot}{{{\boxtimes\,}}}
\nc{\otm}{\overline{\boxtimes}} \nc{\sprod}{\bullet}
\nc{\scs}[1]{\scriptstyle{#1}} \nc{\mrm}[1]{{\rm #1}}
\nc{\margin}[1]{\marginpar{\rm #1}}   
\nc{\dirlim}{\displaystyle{\lim_{\longrightarrow}}\,}
\nc{\invlim}{\displaystyle{\lim_{\longleftarrow}}\,}
\nc{\mvp}{\vspace{0.3cm}} \nc{\tk}{^{(k)}} \nc{\tp}{^\prime}
\nc{\ttp}{^{\prime\prime}} \nc{\svp}{\vspace{2cm}}
\nc{\vp}{\vspace{8cm}} \nc{\proofbegin}{\noindent{\bf Proof: }}
\nc{\proofend}{$\blacksquare$ \vspace{0.3cm}}
\nc{\modg}[1]{\!<\!\!{#1}\!\!>}
\nc{\intg}[1]{F_C(#1)} \nc{\lmodg}{\!
<\!\!} \nc{\rmodg}{\!\!>\!}
\nc{\cpi}{\widehat{\Pi}}
\nc{\sha}{{\mbox{\cyr X}}}  
\nc{\shap}{{\mbox{\cyrs X}}} 
\nc{\shpr}{\diamond}    
\nc{\shp}{\ast} \nc{\shplus}{\shpr^+}
\nc{\shprc}{\shpr_c}    
\nc{\msh}{\ast} \nc{\zprod}{m_0} \nc{\oprod}{m_1}
\nc{\vep}{\epsilon} \nc{\labs}{\mid\!} \nc{\rabs}{\!\mid}
\nc{\sqmon}[1]{\langle #1\rangle}

\nc{\mmbox}[1]{\mbox{\ #1\ }} \nc{\dep}{\mrm{dep}} \nc{\fp}{\mrm{FP}}
\nc{\rchar}{\mrm{char}} \nc{\End}{\mrm{End}} \nc{\Fil}{\mrm{Fil}}
\nc{\Mor}{Mor\xspace} \nc{\gmzvs}{gMZV\xspace}
\nc{\gmzv}{gMZV\xspace} \nc{\mzv}{MZV\xspace}
\nc{\mzvs}{MZVs\xspace} \nc{\Hom}{\mrm{Hom}} \nc{\id}{\mrm{id}}
\nc{\im}{\mrm{im}} \nc{\incl}{\mrm{incl}} \nc{\map}{\mrm{Map}}
\nc{\mchar}{\rm char} \nc{\nz}{\rm NZ} \nc{\supp}{\mathrm Supp}

\nc{\Alg}{\mathbf{Alg}} \nc{\Bax}{\mathbf{Bax}} \nc{\bff}{\mathbf f}
\nc{\bfk}{{\bf k}} \nc{\bfone}{{\bf 1}} \nc{\bfx}{\mathbf x}
\nc{\bfy}{\mathbf y}
\nc{\base}[1]{\bfone^{\otimes ({#1}+1)}} 
\nc{\Cat}{\mathbf{Cat}}

\nc{\detail}{\marginpar{\bf More detail}
    \noindent{\bf Need more detail!}
    \svp}
\nc{\Int}{\mathbf{Int}} \nc{\Mon}{\mathbf{Mon}}
\nc{\rbtm}{{shuffle }} \nc{\rbto}{{Rota-Baxter }}
\nc{\remarks}{\noindent{\bf Remarks: }} \nc{\Rings}{\mathbf{Rings}}
\nc{\Sets}{\mathbf{Sets}} \nc{\wtot}{\widetilde{\odot}}
\nc{\wast}{\widetilde{\ast}} \nc{\bodot}{\bar{\odot}}
\nc{\bast}{\bar{\ast}} \nc{\hodot}[1]{\odot^{#1}}
\nc{\hast}[1]{\ast^{#1}} \nc{\mal}{\mathcal{O}}
\nc{\tet}{\tilde{\ast}} \nc{\teot}{\tilde{\odot}}
\nc{\oex}{\overline{x}} \nc{\oey}{\overline{y}}
\nc{\oez}{\overline{z}} \nc{\oef}{\overline{f}}
\nc{\oea}{\overline{a}} \nc{\oeb}{\overline{b}}
\nc{\weast}[1]{\widetilde{\ast}^{#1}}
\nc{\weodot}[1]{\widetilde{\odot}^{#1}} \nc{\hstar}[1]{\star^{#1}}
\nc{\lae}{\langle} \nc{\rae}{\rangle}
\nc{\lf}{\lfloor}
\nc{\rf}{\rfloor}

\nc{\QQ}{{\mathbb Q}}
\nc{\RR}{{\mathbb R}} \nc{\ZZ}{{\mathbb Z}}

\nc{\cala}{{\mathcal A}} \nc{\calb}{{\mathcal B}}
\nc{\calc}{{\mathcal C}}
\nc{\cald}{{\mathcal D}} \nc{\cale}{{\mathcal E}}
\nc{\calf}{{\mathcal F}} \nc{\calg}{{\mathcal G}}
\nc{\calh}{{\mathcal H}} \nc{\cali}{{\mathcal I}}
\nc{\call}{{\mathcal L}} \nc{\calm}{{\mathcal M}}
\nc{\caln}{{\mathcal N}} \nc{\calo}{{\mathcal O}}
\nc{\calp}{{\mathcal P}} \nc{\calr}{{\mathcal R}}
\nc{\cals}{{\mathcal S}} \nc{\calt}{{\mathcal T}}
\nc{\calu}{{\mathcal U}} \nc{\calw}{{\mathcal W}} \nc{\calk}{{\mathcal K}}
\nc{\calx}{{\mathcal X}} \nc{\CA}{\mathcal{A}}

\nc{\fraka}{{\mathfrak a}} \nc{\frakA}{{\mathfrak A}}
\nc{\frakb}{{\mathfrak b}} \nc{\frakB}{{\mathfrak B}}
\nc{\frakD}{{\mathfrak D}} \nc{\frakF}{\mathfrak{F}}
\nc{\frakf}{{\mathfrak f}} \nc{\frakg}{{\mathfrak g}}
\nc{\frakH}{{\mathfrak H}} \nc{\frakL}{{\mathfrak L}}
\nc{\frakM}{{\mathfrak M}} \nc{\bfrakM}{\overline{\frakM}}
\nc{\frakm}{{\mathfrak m}} \nc{\frakP}{{\mathfrak P}}
\nc{\frakN}{{\mathfrak N}} \nc{\frakp}{{\mathfrak p}}
\nc{\frakS}{{\mathfrak S}} \nc{\frakT}{\mathfrak{T}}
\nc{\frakX}{{\mathfrak X}}
\nc{\BS}{\mathbb{S
}}

\font\cyr=wncyr10 \font\cyrs=wncyr7
\nc{\li}[1]{\textcolor{red}{Li:#1}}
\nc{\yi}[1]{\textcolor{blue}{Yi: #1}}
\nc{\xing}[1]{\textcolor{purple}{Xing:#1}}
\nc{\revise}[1]{\textcolor{red}{#1}}

\nc{\ID}{{\rm I}}\nc{\lbar}[1]{\overline{#1}}\nc{\bre}{{\rm bre}}
\nc{\sd}{\cals}\nc{\rb}{\rm RB}\nc{\A}{\rm A}\nc{\LL}{\rm L}\nc{\tx}{\tilde{X}}
\nc{\col}{\Delta_{RT}}\nc{\mul}{m_{RT}}\nc{\ul}{u_{RT}}\nc{\epl}{\varepsilon_{RT}}
\nc{\hl}{H_{RT}}\nc{\arro}[1]{#1}\nc{\px}{P_{\tx}}\nc{\pw}{P_{\mathfrak{w}}}\nc{\pl}{B_\omega^+}
\nc{\pp}{\pl}\nc{\ppp}[1]{B^+(#1)}\nc{\dw}{\diamond_{\mathfrak{w}}}\nc{\dl}{\diamond_{\rm \ell}}
\nc{\ncshaw}{\sha^{{\rm NC}}_{\Omega}}\nc{\ncshal}{\sha^{{\rm NC}}_{{\rm RT}}}
\nc{\ver}{\rm V}\nc{\ld}{l}\nc{\del}{\Delta_{{\rm \ell}}}\nc{\epsl}{\epsilon_{{\rm \ell}}}
\nc{\uul}{u_{{\rm \ell}}}\nc{\oneh}{\mathbf{1}}\nc{\onew}{\mathbf{1}}
\nc{\etree}{1} \nc{\conc}{m_{RT}}
\nc{\hrtb}{H_{RT}(X\sqcup\Omega)} \nc{\hrts}{H_{RT}(X, \Omega)}\nc{\rts}{\mathcal{T}(X, \Omega)}\nc{\rfs}{\mathcal{F}(X, \Omega)} \nc{\ncshall}{\sha^{{\rm NC}}_{{\rm RT}}} \nc{\ldl}{\leq_{\mathrm{dl}}} \nc{\pla}{B_{\alpha}^{+}} \nc{\plb}{B_{\beta}^{+}}

\nc{\bim}[1]{#1}  \nc{\shaop}{\sha_{\Omega}^{+}}  \nc{\shao}{\sha_{\Omega}}
\nc{\bbim}[2]{#1 #2} \nc{\bbbim}[2]{#1,\, #2} \nc{\RBF}{{\rm RBF}}
\nc{\frbf}{F_{\RBF}} \nc{\shaf}{\ssha_{\tiny{\Omega}}} \nc{\sham}{\diamond_{\tiny{\Omega}}}

\nc{\dnx}{\Delta_n A} \nc{\dx}{\Delta A} \nc{\dgp}{{\rm deg_{P}}}
\nc{\dgt}{{\rm deg_{T}}} \nc{\dg}{{\rm deg}} \nc{\ida}{ID($A$)} \nc{\tu}{\tilde{u}} \nc{\tv}{\tilde{v}}
 \nc{\fbase}{\calb} \nc{\LF}{\mathrm{RF}} \nc{\FFA}{\mathrm{LF}} \nc{\irr}{\mathrm{Irr}}
 \nc{\result}{\bfk\mathrm{Irr}(S_n)}  \nc{\I}{I_{\mathrm{ID},n}^0}
 \nc{\nrs}{\calr_n^\star} \nc{\ii}{\mathrm{I}} \nc{\iii}{\mathrm{II}}
\nc{\intl}{{\rm int}}\nc{\ws}[1]{{#1}}\nc{\deleted}[1]{\delete{#1}}\nc{\plas}{placements\xspace}

\nc{\Id}{\mathrm{Id}} \nc{\Irr}{\mathrm{Irr}}

\begin{document}

\title[Primitive orthogonal idempotents]{Primitive orthogonal idempotents of Brandt semigroup algebras}
%
\author{Yi Zhang} \address{School of Mathematics and Statistics,
Lanzhou University, Lanzhou, Gansu 730000, P.\,R. China}
\email{zhangy2016@lzu.edu.cn}

\author{Jian-Rong Li} \address{School of Mathematics and Statistics, Lanzhou University, Lanzhou 730000, P. R. China.}
\email{lijr@lzu.edu.cn}

\author{Xiao-song Peng} \address{School of Mathematics and Statistics, Lanzhou University, Lanzhou 730000, P. R. China.}
\email{pengxiaosong3@163.com}

\author{Yan-Feng Luo} \address{School of Mathematics and Statistics,
Key Laboratory of Applied Mathematics and Complex Systems,
Lanzhou University, Lanzhou, 730000, P.\,R. China}
         \email{luoyf@lzu.edu.cn}

\date{\today}
\begin{abstract}
A complete set of primitive orthogonal idempotents plays an important role in the representation theory of an associative algebra. In this paper, we construct a complete set of primitive orthogonal idempotents for any finite Brandt semigroup algebra.
\end{abstract}

\subjclass[2010]{
20M25,
16W99, 
}

\keywords{Idempotents; Brandt semigroups; Semigroup algebras}

\maketitle

\tableofcontents

\setcounter{section}{0}

\allowdisplaybreaks

\section{Introduction}

In the past few decades, semigroup algebras are studied intensively, see for example, \cite{Put88}, \cite{Okn91}, \cite{Bro00}, \cite{Ren01}, \cite{JO06}, \cite{Ste06}, \cite{Sal07}, \cite{Sch08}, \cite{Ste08}, \cite{GX09}, \cite{BBBS11}, \cite{DHST11}, \cite{MS12}, \cite{EL15}.

Semigroup algebras are associative algebras. Looking for a complete set of primitive orthogonal idempotents is an important problem in the representation theory of associative algebras.  If $\{e_{i}\}_{i \in I}$ is a complete set of primitive orthogonal idempotents of a finite dimensional algebra $A$, then
\begin{align*}
A\cong\bigoplus_{i}{A{e_{i}}},
\end{align*}
where each $Ae_{i}$ is an indecomposable projective module. They are also used to explicitly compute the quiver, the Cartan matrix, and the Wedderburn decomposition of the algebra, see \cite{Bre10}.

Berg, Bergeron, Bhargava and Saliola \cite{BBBS11} found a complete set of primitive orthogonal idempotents of any $R$-trivial monoid algebra. Denton, Hivert, Schiling and Thiery \cite{DHST11} gave a construction of a set of primitive orthogonal idempotents for any $J$-trivial monoid algebra.

Steinberg \cite{Ste06}, \cite{Ste08} studied the arbitrary finite inverse semigroups $S$ and computed the primitive central idempotents associated to an irreducible representation of $S$.
Let $S$ be a finite inverse semigroup, where $E(S)$ is the semilattice of idempotents of $S$ and the maximal subgroup at $e$ is denoted by $H_{e}$ for all $e\in E(S)$. Suppose that $K$ is a field such that char$(K)\nmid |H_{e}|$.  Let $\mu$ be the M$\ddot{\mathrm{o}}$bius function of $S$ and $\chi$ an irreducible character of $S$ coming from a $\mathcal{D}$-class $D$. Then the primitive central idempotent corresponding to $\chi$ is given by
\begin{align*}
e_{\chi}=\sum_{e\in E(D)}\bigg(\frac{\chi(e)}{|H_{e}|}\sum_{s\in H_{e}}{\chi(s)}\sum_{t\leq s}{t^{-1}}\mu(t,s)\bigg),
\end{align*}
see Theorem 5.1 in \cite{Ste06}.

In this paper, we construct a complete set of primitive orthogonal idempotents for any finite Brandt semigroup algebra. Brandt semigroups are not $R$-trivial and in particular, they are not $J$-trivial. Let $G$ denote any finite group with identity $e$ and $S=B(G,n)$ a Brandt semigroup. Suppose that $G$ has $r$ conjugate classes. Then a complete set of primitive orthogonal idempotents of $\mathbb{C}S$ is (Theorem \ref{complete set for any finite group})
\begin{align*}
\bigg\{ e_{ij}=\frac{\chi_{i}(e)}{|G|}\sum_{g\in G}\chi_{i}(g^{-1})(j,g,j) \mid 1\leq i\leq r, 1\leq j\leq n\bigg\}.
\end{align*}
In general, the primitive orthogonal idempotents are not primitive central idempotents, see Remark \ref{primitive orthogonal idempotents are not central idempotents}.

The paper is structured as follows. In Section $\ref{preliminaries}$, we recall some background information about semigroup algebras, finite group characters, and group algebra codes. In Section $\ref{Primitive Orthogonal Idempotents for $B(G,n)$}$, we construct a complete set of primitive orthogonal idempotents for any finite Brandt semigroup algebra and prove our main theorem (Theorem \ref{complete set for any finite group}) given in Section  $\ref{constructing primitive orthogonal idempotents}$.  In Section $\ref{Example}$, we give some examples of complete sets of primitive orthogonal idempotents of Brandt semigroup algebras.

\section{Preliminaries}\label{preliminaries}
 In this section, we recall the definitions of semigroup algebras, Brandt semigroups, a complete set of primitive orthogonal idempotents, finite group characters, Schur orthogonality relations, which will be used in this paper.
\subsection{Semigroup algebras}  \label{definition of Group algebra}
Let $K$ be an algebraically closed field, $(S, \cdot)$ be a finite semigroup with identity element $e$ and $A$ be a $K$-algebra. The semigroup algebra of $S$ with coefficients in $A$ is the $K$-vector space $AS$ consisting of all the formal sums $\sum_{g \in S}{\lambda_{g}g}$, where $\lambda_g \in A$, with the multiplication defined by the formula
 \begin{equation*}
\left(\sum_{g \in S}{\lambda_{g}g}\right) \cdot \left(\sum_{h \in S}{\mu_{h}h}\right)=\sum_{f=gh \in S}{\lambda_{g}\mu_{h}f}.
\end{equation*}
Then $AS$ is a $K$-algebra and the element $e=1e$ is the identity of $AS$, where $1$ is the identity element of $A$, see \cite{Okn91}. If $A=K$, then the elements $g \in S$ form a basis of $KS$ over $K$.
In this paper we study the Brandt semigroup algebra $\mathbb{C}S$, where $\mathbb{C}$ is the field of complex numbers.
\subsection{Brandt semigroups}
Let $G$ be a group with identity element $e$, and  $\emph{I},\ \Lambda$ be nonempty sets. Let $P=(p_{\lambda i})$ be a $\Lambda\times\emph{I}$ matrix with entries in the 0-group $G^0=G\cup \{0\}$, and suppose that $P$ is regular, in the sense that no row or column of $P$ consists entirely of zeros. Formally,
\begin{equation*}
(\forall i\in \emph{I})\ (\exists \lambda\in \Lambda)\  p_{\lambda i}\neq 0,\\
(\forall \lambda\in \Lambda)\ (\exists i\in \emph{I})\  p_{\lambda i}\neq 0.
\end{equation*}

Let $S=(I\times G\times \Lambda)\cup \{0\}$, and define a multiplication on $S$ by
\begin{equation*}
(i,a,\lambda)(j,b,\mu)=
\begin{cases}
(i,ap_{\lambda j}b,\mu), &\text{if }  p_{\lambda i}\neq 0,\\
0, &\text{if }  p_{\lambda j}=0,
\end{cases}
\end{equation*}
and
\begin{align*}
(i,a,\lambda)0=0(i,a,\lambda)=00=0.
\end{align*}
Then  $S$ is called a completely 0-simple semigroup, see \cite{How95}.

If $P=E$, is an identity matrix, whose diagonal elements are identity of the group $G$ and $\emph{I}=\Lambda$. Then the semigroup $S$ is called a Brandt semigroup, denoted by $B(G,n)$, where $n=|I|$.
\subsection{A complete set of primitive orthogonal idempotents}\label{definition of completes set}
Let $A$ be a $K$-algebra with an identity (denoted by $\mathbf{1}$), see I.4 in \cite{Ass06}. A set of nonzero elements $\{e_i\}_{i \in I}$ of $A$ is called a complete set of primitive orthogonal idempotents for $A$, if it satifies the following four properties:
\begin{enumerate}
\item[(i)] every element $e_i$ is idempotent, namely, ${e_i}^2=e_i$ for all $i \in I$;

\item[(ii)] each of the two elements are orthogonal: $e_ie_j=e_je_i=0$ for all $i, j \in I$ with $i\neq j$;

\item[(iii)] every element $e_i$ is primitive: $e_i$ cannot be written as a sum, that is, if $e_i=x+y$, then $x=0$ or  $y=0$, where $x$ and $y$ are orthogonal idempotents in $A$;

\item[(iv)] the set $\{e_i\}_{i \in I}$ is complete: $\sum_{i \in I}{e_i=\mathbf{1}}$.
\end{enumerate}

\begin{remark}[Remark 3.2, \cite{BBBS11}]\label{maximal}
If $\{e_i\}_{i \in I}$ is a maximal set of nonzero elements satisfying conditions $\mathrm{(i)}$ and $\mathrm{(ii)}$, then $\{e_i\}_{i \in I}$ is a complete set of primitive orthogonal idempotents (that is, $\mathrm{(iii)}$ and $\mathrm{(iv)}$ also hold).
\end{remark}

\subsection{Finite group characters}
Let $V$ be a finite dimensional vector space over a field $K$ and $\rho : G\rightarrow GL(V)$ a representation of a group $G$ on $V$. The character of $\rho$ is the function $\chi_{\rho}: G\rightarrow K$ given by
\begin{align*}
\chi_{\rho}(g)=\text{Tr}(\rho(g)),
\end{align*}
where $\text{Tr}$ is the trace.

A character $\chi_{\rho}$ is called irreducible if $\rho$ is an irreducible representation. The number of conjugacy classes of $G$ is equal to the number of irreducible characters of $G$ and equals the number of isomorphism classes of irreducible $KG$-modules. The degree of the character $\chi$ is the dimension of $\rho$ and this is equal to the value $\chi(e)$, see \cite{Isa76}.

\subsection{Schur orthogonality relations}
Schur orthogonality relations, see \cite{Isa76}, express a central fact about representations of finite groups. The space of complex valued class functions of a finite group $G$ has a natural inner product
\begin{align*}
\langle\alpha,\beta\rangle=\frac{1}{|G|}\sum_{g\in G}\alpha(g)\overline{\beta(g)},
\end{align*}
where $\overline{\beta(g)}$ is the complex conjugate of the value of $\beta$ on $g$.

With respect to this inner product, the irreducible characters form an orthogonal basis for the space of class functions, and this yields the orthogonality relation for the rows of the character table
\begin{align*}
\langle\chi_{i},\chi_{j}\rangle=
\begin{cases}
0,  &\text{if }i\neq j, \\
1,  &\text{if }i=j.
\end{cases}
\end{align*}
For $g, h \in G$, the orthogonality relation for columns is given by
\begin{align*}\label{orthogonality relation for columas}
\sum_{\chi_{i}}\chi_{i}(g)\overline{\chi_i(h)}=
\begin{cases}
|C_{G}(g)|, &\text{if }g, h \text{ are conjugate}, \\
0,          &\text{otherwise},
\end{cases}
\end{align*}
where $|C_{G}(g)|$ denotes the cardinality of the centralizer of $g$.

\section{A complete set of primitive orthogonal idempotents for $B(G,n)$}\label{Primitive Orthogonal Idempotents for $B(G,n)$}
In this section, we construct a complete set of primitive orthogonal idempotents of any finite Brandt semigroup algebra.

\subsection{Constructing primitive orthogonal idempotents}\label{constructing primitive orthogonal idempotents}
Suppose that $G$ is any finite group with identity $e$ and has $r$ conjugacy classes. Let $S=B(G,n)$ be a finite Brandt semigroup and $\chi_{i}$ an irreducible complex character of $G$. For $1\leq {i} \leq r$, ${1\leq j\leq n}$, we define
\begin{align}
e_{ij}=\frac{\chi_{i}(e)}{|G|}\sum_{g\in G}\chi_{i}(g^{-1})(j,g,j).
\end{align}
We note that the element $\sum_{i=1}^{n}{(i,e,i)}$ is the identity of Brandt semigroup algebra $\mathbb{C}S$, denoted by $\mathbf{1}$.

The following lemmas are needed.
\begin{lemma}[Corollary 2.7, \cite{Isa76}]\label{irreducible complex character}
Let $G$ be a finite group with the identity $e$ and $\mathrm{Irr}(G)$ the set of all irreducible complex character of $G$. Then $|\mathrm{Irr}(G)|$ equals the number of conjugacy classes of $G$ and
\begin{align*}
\sum_{\chi \in \mathrm{Irr}(G)}\chi(e)^2=|G|.
\end{align*}
\end{lemma}
By Schur orthogonality relations, we have the following lemma.
\begin{lemma} [Theorem 2.13, \cite{Isa76}]\label{orthogonal relation}
Let $G$ be a finite group with the identity $e$ and $\chi_{i}$ an irreducible complex character of $G$. Then the following holds for every $h\in G$,
\begin{align*}
\frac{1}{|G|}\sum_{g\in G}\chi_{i}(gh)\chi_{j}(g^{-1})=\delta_{ij}\frac{\chi_{i}(h)}{\chi_{i}(e)},
\end{align*}
where
\begin{align*}
\delta_{ij}=
\begin{cases}
1 &\text{if } i=j,\\
0 &\text{if } i\neq j.
\end{cases}
\end{align*}
\end{lemma}

\subsection{Main results}
The following theorem is our main result.
\begin{theorem}\label{complete set for any finite group}
Let $G$ denote any finite group with identity $e$ and $S=B(G,n)$ a Brandt semigroup. Suppose that $G$ has $r$ conjugacy classes. Then the elements
\begin{align*}
e_{ij}=\frac{\chi_{i}(e)}{|G|}\sum_{g\in G}\chi_{i}(g^{-1})(j,g,j), \label{complete set for any CS}
\end{align*}
where ${1\leq i\leq r}$, ${1\leq j\leq n}$, form a complete set of primitive orthogonal idempotents of the Brandt semigroup algebra $\mathbb{C}S$.
\end{theorem}
\begin{proof}
Firstly, we prove the set $\{e_{ij}\}_{1\leq i \leq r,\ 1\leq j \leq n}$ is complete.
Since $G$ has $r$ conjugacy classes, $G$ has $r$ irreducible representations. For any $1\leq j \leq n$, we have
\begin{align*}\nonumber
\sum_{i=1}^{r}e_{ij}&=\frac{1}{|G|}\sum_{i=1}^{r}\sum_{g\in G}\chi_{i}(e)\chi_{i}(g^{-1})(j,g,j) \\ \nonumber
&=\frac{1}{|G|}\sum_{i=1}^{r}\chi_{i}(e)\big(\chi_{i}(e^{-1})(j,e,j)+\sum_{g\in G\backslash\{e\}}\chi_{i}(g^{-1})(j,g,j)\big) \\ \nonumber
&=\frac{1}{|G|}\sum_{i=1}^{r}\chi_{i}(e)\chi_{i}(e^{-1})(j,e,j)+\frac{1}{|G|}\sum_{g\in G\backslash\{e\}}\sum_{i=1}^{r}\chi_{i}(e)\chi_{i}(g^{-1})(j,g,j).\nonumber
\end{align*}
Obviously, $e$ and $g\in G\backslash\{e\}$ are not conjugate. According to Schur orthogonality relations (\ref{orthogonality relation for columas}), we have
\begin{align*}
\sum_{i=1}^{r}\chi_{i}(e)\chi_{i}(g)=0,\nonumber
\end{align*}
where $g\neq e$.

Hence
\begin{align*}
\sum_{i=1}^{r}\chi_{i}(e)\chi_{i}(g^{-1})=0.\nonumber
\end{align*}
Thus
\begin{align*}\nonumber
\sum_{i=1}^{r}e_{ij}=\frac{1}{|G|}\sum_{i=1}^{r}\chi_{i}(e)\chi_{i}(e^{-1})(j,e,j)
=\frac{1}{|G|}\sum_{i=1}^{r}(\chi_{i}(e))^{2}(j,e,j).
\end{align*}
By Lemma \ref{irreducible complex character}, we have
\begin{align*}
\sum_{i=1}^{r}(\chi_{i}(e))^{2}=|G|.\nonumber
\end{align*}
Then
\begin{align*}
\sum_{i=1}^{r}e_{ij}=(j,e,j).\nonumber
\end{align*}
Thus
\begin{align*}
\sum_{j=1}^{n}\sum_{i=1}^{r}e_{ij}=\sum_{j=1}^{n}(j,e,j)=\mathbf{1}.\nonumber
\end{align*}

Secondly, we prove that the elements of $\{e_{ij}\}_{1\leq i \leq r,\ 1\leq j \leq n}$ are pairwise orthogonal.
For any $j_{1}\neq j_{2}$, $1\leq j_{1}, j_{2} \leq n$, $1\leq i_{1}, i_{2} \leq r$,
\begin{align*}
e_{i_{1}j_{1}}e_{i_{2}j_{2}}&=\left(\frac{\chi_{i_{1}}(e)}{|G|}\sum_{g' \in G}\chi_{i_{1}}(g'^{-1})(j_{1},g',j_{1})\right)\left(\frac{\chi_{i_{2}}(e)}{|G|}\sum_{g \in G}\chi_{i_{2}}({g}^{-1})(j_{2},g,j_{2})\right)\\
&=0.
\end{align*}
For any $j_{1}=j_{2}$, $i_{1}\neq i_{2}$,
\begin{align*}
e_{i_{1}j_{1}}e_{i_{2}j_{2}}&=\left(\frac{\chi_{i_{1}}(e)}{|G|}\sum_{g' \in G}\chi_{i_{1}}(g'^{-1})(j_{1},g',j_{1})\right)\left(\frac{\chi_{i_{2}}(e)}{|G|}\sum_{g \in G}\chi_{i_{2}}({g}^{-1})(j_{2},g,j_{2})\right)\\
&=\frac{\chi_{i_{1}}(e)\chi_{i_{2}}(e)}{|G|^2}\sum_{{g'}\in G}{\sum_{g\in G}}{\chi_{i_{1}}({g'}^{-1})}{\chi_{i_{2}}({g}^{-1})}(j_{1},g'g,j_{1})\\
&=\frac{\chi_{i_{1}}(e)\chi_{i_{2}}(e)}{|G|^2}\sum_{g''\in G}\sum_{g\in G}\chi_{i_{1}}({g{g''}^{-1}})\chi_{i_{2}}({g}^{-1})(j_{1},g'',j_{1}).
\end{align*}
By Lemma \ref{orthogonal relation},
\begin{align*}
\sum_{g\in G}\chi_{i_{1}}({g{g''}^{-1}})\chi_{i_{2}}({g}^{-1})=\delta_{i_{1}i_{2}}\frac{\chi_{i_{1}}({g''}^{-1})}{\chi_{i_{1}}(e)}.
\end{align*}
Then, for any $i_{1}\neq i_{2}$,
\begin{align*}
e_{i_{1}j_{1}}e_{i_{2}j_{1}}=0.
\end{align*}
Hence, the elements of $\{e_{ij}\}_{1\leq i \leq r,\ 1\leq j \leq n}$ are pairwise orthogonal.

In the following,  we prove that each element of $\{e_{ij}\}_{1\leq i \leq r,\ 1\leq j \leq n}$ is idempotent.
Let $e_{st}$ be any element of the set $\{e_{ij}\}_{1\leq i \leq r,\ 1\leq j \leq n}$, where $1\leq s \leq r$, $1\leq t\leq n$. Then
\begin{align*}
{e_{st}}^{2}=e_{st}\left(\sum_{j=1}^{n}\sum_{i=1}^{r}e_{ij}\right)=e_{st}\mathbf{1}=e_{st}.
\end{align*}

Finally, we prove that each element of $\{e_{ij}\}_{1\leq i \leq r,\ 1\leq j \leq n}$ is primitive. Suppose that $f\in \mathbb{C}S$, $f^2=f$, $f\neq0$, $f\notin \{e_{ij}\}_{1\leq i \leq r,\ 1\leq j \leq n} $, and $f$ is orthogonal to every $e_{ij}$, $1\leq i \leq r$, $1\leq j \leq n$. Then
\begin{align*}
f=f\mathbf{1}=f\left(\sum_{j=1}^{n}\sum_{i=1}^{r}e_{ij}\right)=0.
\end{align*}
This contradicts the assumption that $f\neq 0$. Therefore $\{e_{ij}\}_{1\leq i \leq r,\ 1\leq j\leq n}$ is a maximal set of nonzero elements satisfying
conditions (i) and (ii) in Section \ref{definition of completes set}. By Remark \ref{maximal}, each element of $\{e_{ij}\}_{1\leq i \leq r,\ 1\leq j \leq n}$ is primitive.
\end{proof}

\begin{corollary}\label{complete set for CS}
Let $G=\langle a\mid a^{k}=e\rangle$ be a  finite cyclic group and $S$ a Brandt semigroup with $|\Lambda|=|I|=n$. Then for ${1\leq p\leq k}$, $0\leq q \leq n-1$, the elements

\begin{align*}
e_{p+qk(n+1)}={\frac {1}{k}}{\sum_{j=1}^{k}{\omega^{-jp}{(q+1,a^j,q+1)}}},
\end{align*}
where $\omega=e^{\frac{2\pi i}{k}}$, $i=\sqrt{-1}$, form a complete set of primitive orthogonal idempotents of the Brandt semigroup algebra $\mathbb{C}S$.
\end{corollary}
\begin{proof}
This follows directly from Theorem \ref{complete set for any finite group}.
\end{proof}
\begin{remark}\label{primitive orthogonal idempotents are not central idempotents}
Let $S=B(G,n)$ be a finite Brandt semigroup. If $n=1$, the primitive orthogonal idempotents of $\mathbb{C}S$ are precisely primitive central idempotents.
If $n\geq 2$, in general, the primitive orthogonal idempotents are not primitive central idempotents. We can see the following example.

Let $G=\{e\}$, $I=\{1,2\}$. Then we have the Brandt semigroup
\begin{align*}
S=B_{2}=\langle a,b \mid a^2=b^2=0,aba=a,bab=b \rangle. \nonumber
\end{align*}
We also have the identity
\begin{align*}
\mathbf{1}=\sum_{i=1}^{2}(i,e,i). \nonumber
\end{align*}
By Theorem \ref{complete set for CS}, the elements
\begin{align*}
e_{1}=(1,e,1),\quad e_{4}=(2,e,2)\nonumber
\end{align*}
form a complete set of primitive orthogonal idempotents in the semigroup algebra $\mathbb{C}S$.

Let $\mathbb{C}S=\{\sum_{i=1}^{4}k_{i}s_{i}\mid k_{i}\in \mathbb{C},s_{i}\in S \}$. Then we have
\begin{align*}\nonumber
e_{1}\mathbb{C}S&=\big\{k_{1}(1,e,1)+k_{2}(1,e,2)\mid k_{1},k_{2} \in \mathbb{C}\big\},\\\nonumber
\mathbb{C}Se_{1}&=\big\{k_{1}(1,e,1)+k_{3}(2,e,1)\mid k_{1},k_{2} \in \mathbb{C}\big\}.
\end{align*}
Obviously,
\begin{align*}
e_{1}\mathbb{C}S\neq\mathbb{C}Se_{1}.\nonumber
\end{align*}
Then the primitive orthogonal idempotents are not primitive central idempotents in the semigroup algebra $\mathbb{C}S$.
\end{remark}

\section{Examples}\label{Example}
In this section, we give some examples of complete sets of primitive orthogonal idempotents of Brandt semigroup algebras.
\begin{example}\label{Example1}
Let $G_3=\langle a\mid a^3=e\rangle$ be a cyclic group and $\Lambda =I=\{2\}$. Then the elements
\begin{align*}
e_1 & =\frac{1}{3}\sum_{j=1}^{3}{\omega^{-3j}{(1,a^j,1)}}={\frac{1}{3}}(1,e,1)+{\frac{1}{3}}(1,a,1)+{\frac{1}{3}}(1,a^2,1),\\
e_2 & =\frac{1}{3}\sum_{j=1}^{3}{\omega^{-2j}{(1,a^j,1)}}={\frac{1}{3}}(1,e,1)- (\frac{1}{6} - \frac{\sqrt{3}\, \mathrm{i}}{6})(1,a,1) - (\frac{1}{6} + \frac{\sqrt{3}\, \mathrm{i}}{6})(1,a^2,1),\\
e_3 & =\frac{1}{3}\sum_{j=1}^{3}{\omega^{-j}{(1,a^j,1)}}={\frac{1}{3}}(1,e,1)- (\frac{1}{6} + \frac{\sqrt{3}\, \mathrm{i}}{6})(1,a,1) - (\frac{1}{6} - \frac{\sqrt{3}\, \mathrm{i}}{6})(1,a^2,1),\\
e_4 & =\frac{1}{3}\sum_{j=1}^{3}{\omega^{-3j}{(2,a^j,2)}}={\frac{1}{3}}(2,e,2)+{\frac{1}{3}}(2,a,2)+{\frac{1}{3}}(2,a^2,2),\\
e_5 & =\frac{1}{3}\sum_{j=1}^{3}{\omega^{-2j}{(2,a^j,2)}}={\frac{1}{3}}(2,e,2)- (\frac{1}{6} - \frac{\sqrt{3}\, \mathrm{i}}{6})(2,a,2) - (\frac{1}{6} + \frac{\sqrt{3}\, \mathrm{i}}{6})(2,a^2,2),\\
e_6 & =\frac{1}{3}\sum_{j=1}^{3}{\omega^{-j}{(2,a^j,2)}}={\frac{1}{3}}(2,e,2)- (\frac{1}{6} + \frac{\sqrt{3}\, \mathrm{i}}{6})(2,a,2) - (\frac{1}{6} - \frac{\sqrt{3}\, \mathrm{i}}{6})(2,a^2,2),
\end{align*}
where $w=e^{\frac{2\pi i}{3}}=\cos \frac{2\pi}{3}+i{\sin \frac{2\pi}{3}}$, $i=\sqrt{-1}$, form a complete set of primitive orthogonal idempotents of $\mathbb{C}B(G_3,1)$.
\end{example}
\begin{example}
Let $G=\langle a\rangle\times \langle b\rangle\cong C_{2}\times C_{2}$ be an abelian group of order $4$. The character table of $G$ is as follows.
\begin{table}[H]
\begin{tabular}{|c|c|c|c|c|}
  \hline
    $G$ & $e$ & $a$ & $b$ & $ab$ \\
  \hline
   $\chi_{1}$ & 1 & 1 & 1 & 1 \\
  \hline
  $\chi_{2}$ & 1 & -1 & 1 & -1 \\
  \hline
  $\chi_{3}$ & 1 & 1 & -1 & -1 \\
  \hline
  $\chi_{4}$ & 1 & -1 & -1 & 1 \\
\hline
\end{tabular}
\caption{Character table of abelian group $G$.}
\end{table}
Let $\Lambda=I=\{1\}$. Then the Brandt semigroup algebra $\mathbb{C}B(G,1)$ has four primitive orthogonal idempotents:
\begin{align*}\nonumber
e_{11}=\frac{1}{4}\big((1,e,1)+(1,a,1)+(1,b,1)+(1,ab,1)\big),\\ \nonumber
e_{21}=\frac{1}{4}\big((1,e,1)-(1,a,1)+(1,b,1)-(1,ab,1)\big),\\ \nonumber
e_{31}=\frac{1}{4}\big((1,e,1)+(1,a,1)-(1,b,1)-(1,ab,1)\big),\\ \nonumber
e_{41}=\frac{1}{4}\big((1,e,1)-(1,a,1)-(1,b,1)+(1,ab,1)\big).
\end{align*}
\end{example}

\begin{example}
Let $G=S_3$. This group has $6$ elements:
\begin{align*}\nonumber
1, \underbrace{(12), (13), (23)}, \underbrace{(123), (132)},
\end{align*}
where $\underbrace{}$ means the elements are conjugate. There are three conjugacy classes.

Next, we have the character table of $S_3$.
\begin{table}[H]
\begin{tabular}{|c|c|c|c|}
  \hline
    $S_{3}$ & 1 & (12) & (123) \\
  \hline
   $\chi_{1}$ & 1 & 1 & 1 \\
  \hline
  $\chi_{2}$ & 1 & -1 & 1  \\
  \hline
  $\chi_{3}$ & 2 & 0 & -1  \\
  \hline
\end{tabular}
\caption{Character table of $S_3$.}
\end{table}
Let $\Lambda=I=\{1,2\}$. Then the Brandt semigroup algebra $\mathbb{C}B(S_3,2)$ has six primitive orthogonal idempotents:
\begin{align*}\nonumber
e_{11}&=\frac{1}{6}\big((1,1,1)+(1,(123),1)+(1,(132),1)+(1,(12),1)+(1,(13),1)+(1,(23),1)\big),\\ \nonumber
e_{21}&=\frac{1}{6}\big((1,1,1)+(1,(123),1)+(1,(132),1)-(1,(12),1)-(1,(13),1)-(1,(23),1)\big), \\ \nonumber
e_{31}&=\frac{1}{6}\big(2^2(1,1,1)-2(1,(123),1)-2(1,(132),1)\big), \\ \nonumber
e_{12}&=\frac{1}{6}\big((2,1,2)+(2,(123),2)+(2,(132),2)+(2,(12),2)+(2,(13),2)+(2,(23),2)\big),\\ \nonumber
e_{22}&=\frac{1}{6}\big((2,1,2)+(2,(123),2)+(2,(132),2)-(2,(12),2)-(2,(13),2)-(2,(23),2)\big), \\ \nonumber
e_{32}&=\frac{1}{6}\big(2^2(2,1,2)-2(2,(123),2)-2(2,(132),2)\big).  \nonumber
\end{align*}
These primitive orthogonal idempotents are obtained from Theorem \ref{complete set for any finite group}. For example,
\begin{align*}\nonumber
e_{31}&=\frac{\chi_{3}(1)}{6}\big(\chi_{3}(1)^{-1}(1,1,1)+\chi_{3}(123)^{-1}(1,(123),1)
+\chi_{3}(132)^{-1}(1,(132),1)\big) \\ \nonumber
&=\frac{1}{6}\big(2^2(1,1,1)-2(1,(123),1)-2(1,(132),1)\big). \\ \nonumber
\end{align*}

\end{example}

\medskip

\noindent {\bf Acknowledgments}: This work was supported by the National Natural Science Foundation
of China (Grant No.\@ 11771191, 11501267, 11401275).

\end{document}